\documentclass[11pt]{amsart}
\usepackage[mathscr]{eucal}
\usepackage{times}
\usepackage{amsmath,amssymb,latexsym,amscd}   
\usepackage{hyperref}
\hypersetup{colorlinks=true,linkcolor=blue,citecolor=red,linktocpage=true}
\usepackage[english]{babel}
\usepackage{amsfonts}
\usepackage{amssymb}
\usepackage{graphicx}
\usepackage[all,cmtip]{xy}
\usepackage{amsthm}
\usepackage{enumerate}
\usepackage[us,12hr]{datetime} 
\author[Medina,\\ Zald\'ivar, \\Sandoval]{Mauricio Gabriel Medina-B\'arcenas \\ Luis Angel Zald\'ivar-Corichi \\ Martha Lizbeth Shaid Sandoval-Miranda}
\title{ A generalization of quantales with applications to modules and rings }

\theoremstyle{plain}
\newtheorem{thm}{\protect\theoremname}[section]
\theoremstyle{definition}
\newtheorem{rem}[thm]{\protect\remarkname}
\theoremstyle{definition}
\newtheorem{example}[thm]{\protect\examplename}
\theoremstyle{definition}

\theoremstyle{plain}
\newtheorem{cor}[thm]{\protect\corollaryname}
\theoremstyle{plain}
\newtheorem{lemma}[thm]{\protect\lemmaname}
\theoremstyle{plain}
\newtheorem{prop}[thm]{\protect\propositionname}
\theoremstyle{definition}
\newtheorem{dfn}[thm]{\protect\definitionname}

\newcommand{\RMod}{{R{-}Mod}}
\newcommand{\Hom}{\mathrm{Hom}}

\newcommand{\R}{{R}}

\newcommand{\Rpr}{\R{-}pr}

\newcommand{\I}{\underline{1}}

\makeatother

\providecommand{\lemmaname}{Lemma}
\providecommand{\corollaryname}{Corollary}
\providecommand{\propositionname}{Proposition}
\providecommand{\definitionname}{Definition}
\providecommand{\theoremname}{Theorem}
\providecommand{\remarkname}{Remark}
\providecommand{\examplename}{Example}

\begin{document}

\begin{abstract}
In this paper we introduce a lattice structure as a generalization of meet-continuous lattices and quantales. We develop a point-free approach to these new lattices and apply these results to $R$-modules. In particular, we give a module counterpart of the well known result that in a commutative ring the set of semiprime ideals, that is, radical ideals is a frame. 
\end{abstract}
\maketitle

\section{Introduction}\label{sec:sec1}
In general, for the study of rings and modules, the complete lattices of submodules play an important role, there are various forms for the examination of these lattices, like via pure lattice theory or in a categorical point of view. In the study of rings there are other lattice-structures in play, like preradicals, torsion theories, linear filters etc. In the eye of this doctrine, we introduce a framework of lattice structure theory to analyze the submodules of a given module and preradicals on a ring, in particular we specialize in the lattice of submodules and we give some characterizations of these.      

In \cite{H} and \cite{Ro}, the authors introduce some order-lattice structures for studying  rings, topological spaces and more. In particular, in \cite{H}, the author considers lattices equipped with a product in which an absolute distribution holds with respect to take suprema and certain comparisons between the $\wedge$ and this product must be hold. He calls these kind of lattices \emph{carries}, an by the definition these ones are particular examples of \emph{quantales}. From this theory one can describe interesting phenomena of ring theory like \emph{radical theory}. This is one of our motivations in the scope of a more general situation, which is determinate radicals in the sense of rings but now for modules. Other important structures for studying rings and modules are \emph{idioms} and \emph{frames}. Recall that idioms are complete modular meet-continuous lattices, therefore they are a kind of generalization of frames. Another generalization of frames are quantales. With this in mind, one of our first goals is to introduce a generalization of idioms and quantales in which both of them are particular cases.

This paper is organized as follows: in section \ref{sec:sec2} we give a brief introduction of the notation and some examples that motivate our study, in section \ref{sec:sec3} we introduce a general framework for the study of quantales and idioms, we generalize some results of localic nuclei in quantale theory. In last section, we apply these notions to the idiom of submodules and we generalize some classical structure theorems in commutative ring theory  .

\section{Preliminaries and background material}\label{sec:sec2}

 Throughout this paper,  we will work with complete lattices in the usual sense. Recall that a $\bigvee$-semilattice is an structure $(A,\leq,\bigvee,0)$, where $(A,\leq)$ is a partial ordered set with bottom element $0$ and a binary operation $\vee$ such that $x\vee y\leq z$ if and only if $x\leq z \text{ and } y\leq z$, for all, $a,y,z\in A.$ Also all the $\bigvee$-semilattices we considered here are complete with respect $\vee$. A morphisms of $\bigvee$-semilattices, $f\colon A\rightarrow B$ is a monotone function that preserves arbitrary suprema, that is,  $f[\bigvee X]=\bigvee f[X]$ for all $X\subseteq A$. Then, by the adjoint functor theorem it follows that   any morphism of $\bigvee$-semilattices, namely $f\colon A\rightarrow B$, has a right adjoint $f_{*}\colon B\rightarrow A.$ Denoting $f=f_{*},$ the composition $d=f_{*}f^{*}\colon A\rightarrow A$ is a closure operation, that is, $d$ is a monotone function, $a\leq d(a)$ and $dd=d.$ It can be shown that any of these operators rises naturally in this way.
    
Here our analysis is concentrated in certain kind of $\bigvee$-semilaticces. Those are the following:

A lattice $A$ is {\em meet-continuous} if 
\[a\wedge \left(\bigvee X\right)=\bigvee\left\{a\wedge x\mid x\in X\right\}\leqno({\rm IDL}),\]
for all $a\in A$ and $X\subseteq A$ any directed set. Here, $X$ is {\em directed} if it is non-empty and for each $x,y\in X$ there is some $z\in X$ with $x\leq z$ and $y\leq z$.

A lattice $A$ is \emph{modular} if 
$(a\vee c)\wedge b=a\vee(c\wedge b),$ for all $a,b,c\in A$ such that $a\leq b$.

An \emph{idiom} is a meet-continuous modular lattice. An important class of idioms are the well known \emph{frames}. Recall that a complete lattice $A$ is a \emph{frame} if $$a\wedge \left(\bigvee X\right)=\bigvee\left\{a\wedge x\mid x\in X\right\}\leqno({\rm FDL})$$ holds for every $a\in A$ and every $X\subseteq A$.  

On the other hand,  for any lattice  $A,$ an \emph{implication}  is a two placed operation $(\_\succ \_)$ given by $x\leq (a\succ b)\Leftrightarrow x\wedge b\leq a$, for all $a,b\in A.$ 

The distributivity law (FDL) characterizes frames as follows.

\begin{prop}\label{p1}
A complete lattice $A$ is a frame if and only if $A$ has an implication.
\end{prop}

This implies that in a frame $A$ any element $a\in A$ has a \emph{negation} $(a\succ \underline{0})$ or simply $\neg a$. A proof of \ref{p1} and general background about idioms and frames can be viewed in  \cite{H2}, \cite{Pi} and\cite{H3}.
Some interesting examples of idioms and frames are: 
\begin{example}\noindent
\begin{enumerate}
\item  Given an associative unital ring $R$ consider the category of left $R$-modules, $\RMod$ and for every $M\in\RMod$, denote by $\Lambda(M)$ the left submodules of $M.$ It is well known that $\Lambda(M)$ is a complete meet-continuous modular lattice, whence an idiom.
\item Consider any topological space $S$ and denote by $\mathcal{O}(S)$ the open subsets of $S.$ This is the generic non-trivial example of a frame.
\end{enumerate}
\end{example}

So any morphism of idioms (in particular of frames) as subcategories of the category of $\bigvee$-semilattices gives closure operators, we shall summarize systematically the definition of these operators.

An \emph{inflator} in a idiom $A$ is a function $d\colon A\rightarrow A$ such that $x\leq d(x)$ and $ x\leq y \Rightarrow d(x)\leq d(y)$. 
A \emph{pre-nucleus} $d$ on $A$ is an inflator such that $d(x\wedge y)=d(x)\wedge d(y)$. 
A \emph{stable} inflator on $A$ is an inflator such that $d(x)\wedge y\leq d(x\wedge y)$ for all $x,y\in A$.   
Let us denote by $I(A)$ the set of all inflators on $A$ and let $P(A)$ be the set of all pre-nuclei and $S(A)$ the set of all stable inflators. Clearly, $P(A)\subseteq S(A)\subseteq I(A)$. 
Note that from the definition of inflator, the composition of any two inflators is again an inflator. In fact, $I(A)$ is a poset with the order given for $d,d'\in I(A)$ by $d\leq d'\Leftrightarrow d(a)\leq d'(a)$, for all $a\in A$. The identity $d_{\underline{0}}$ of $A$  and the constant function $\bar{d}(a)=\bar{1}$ for all $a\in A$, are inflators.  Furthermore, these ones are pre-nuclei. 

For any two inflators $d,d'$ on $A$ and for all $a\in A$ we have $a\leq d(a)$. From the monotonicity of $d'$ it follows that $d'(a)\leq d'(d(a)),$ and again we have $d'(d(a))\geq d(a)$. Therefore,  if $A$ is an idiom, then for any two inflators $d,d'\in I(A)$ we have that $d'\vee d\leq d'd$. It follows the next:

\begin{lemma}\label{p2} If  $A$ is an idiom and $d$, $d'$, $k$ are inflators over $A$, then:
\begin{enumerate}
\item If $d\leq d'$, then $kd\leq kd'$ and $dk\leq d'k$.
\item $kd'\vee kd\leq k(d'\vee d)$ and $ k(d'\wedge d)\leq kd'\wedge kd$. 
\item Moreover, if $\mathcal{D}\subseteq I(A)$ is non empty, then:
\begin{enumerate}[(a)]
\item $(\bigvee \mathcal{D})k=\bigvee\left\{dk\mid d\in \mathcal{D}\right\}$,
\item $(\bigwedge \mathcal{D})k=\bigwedge\left\{dk\mid d\in \mathcal{D}\right\}$.
\end{enumerate}
\end{enumerate}
\end{lemma}

The properties of the composition of inflators gives a crucial construction:

Given an inflator $d\in I(A)$, let $d^{0}:=d_{\underline{0}}$,  $d^{\alpha+1}:=d\circ d^{\alpha}$ for a non-limit ordinal $\alpha$, and let  $d^{\lambda}:=\bigvee\left\{d^{\alpha}\mid \alpha<\lambda\right\}$ for a limit ordinal $\lambda$.
These are inflators, and from the comparison $d\vee d'\leq dd'$, we have a chain of inflators 
$$d\leq d^{2}\leq d^{3}\leq\ldots \leq d^{\alpha}\leq\ldots.$$ 
By a cardinality argument, there exists an ordinal $\gamma$ such that $d^{\alpha}=d^{\gamma}$, for $\alpha\geq \gamma$. In fact, we can choose $\gamma$ the least of these ordinals, say $\infty$. Thus, $d^{\infty}$ is an inflator which not only satisfies $d\leq d^{\infty}$ but also  $d^{\infty}d^{\infty}=d^{\infty}$, that is, $d^{\infty}$ is an idempotent or  closure operator over $A$. Denote by $C(A)$ the set of all closure operators. This is a poset in which the infimum of a set of closure operators is again a closure operator. Hence, it is a complete lattice, and   the construction we just made defines an operator $(\_)^{\infty}\colon I(A)\rightarrow C(A)$. It is clear that this operation (in the second level) is inflatory and monotone.
 The supremum in $C(A)$ for an arbitrary family of closure operators can be computed as follows:  first, take a non-empty subset of closure operators $\mathcal{C}$ 
on $A$.  By the observations of the supremum, $\bigvee \mathcal{C}$ is an inflator, so we can apply the construction above and obtain a closure operator $(\bigvee \mathcal{C})^{\infty},$  this is the supremum of the family $\mathcal{C}$ in $C(A)$.  For an inflator $j$ over an idiom $A$, we say that $j$ is a \emph{nucleus} if is an idempotent pre-nucleus.  By induction and  distributivity on idioms, it follows that if $A$ is an idiom, then: 
\begin{enumerate}
 
\item If $f$ is a pre-nucleus on $A$, then $f^{\alpha}$ is a pre-nucleus. In particular, $f^{\infty}$ is a nucleus.
\item If $f$ is a stable inflator, then each $f^{\alpha}$ is a pre-nucleus, and for limit ordinals $\lambda$, $f^{\lambda}$ is a pre-nucleus. In particular, $f^{\infty}$ is a nucleus.
\end{enumerate}
One of the most important observations in the inflator theory is:

\begin{thm}\label{p3}
For any idiom $A$, denote by $N(A)$ the set of all nuclei on $A$ then, the complete lattice $N(A)$ is a frame.
\end{thm}
Most of the structural analysis of idioms is related with the frame $N(A)$, as ranking techniques, dimensions and other point free-theoretical constructions. In particular, every nucleus $j$ on $A$ determines a quotient of the idiom, the set of all fixed points of $j$, $A_{j}$ together with an idiom morphism $j^{*}\colon A\rightarrow A_{j}$ given by $j^{*}(a)=j(a)$.  

There are other $\bigvee$-semilattices structures that we need to mention.

Let $Q$ be a $\bigvee$-semilattice, $Q$ is a \emph{quantale} if $Q$ has a binary associative operation $\cdot: Q\times Q\rightarrow Q$ such that
 \[l\left(\bigvee X\right)=\bigvee\left\{lx\mid x\in X\right\} , \mbox{ and } \leqno({\rm LQ}) \] 
 \[\left(\bigvee X\right)r=\bigvee\left\{xr\mid x\in X\right\}\leqno({\rm RQ})\] 

\noindent hold for all $l, r\in Q$ and $X\subseteq Q.$ 

Notice that  if In the previous definition we  assume  meets instead of joins, then we obtain the dual quantale with respect to meets.

In case that only {\rm (LQ)} (respectely {\rm (RQ)}) holds, some authors just say that $Q$ is a  left quantale (respectely right quantale).

Note that any frame is a quantal, with the operation $\wedge$, in fact:

\begin{prop}\label{p4}
Let $(Q, \leq, \bigvee, 0, \cdot)$ be a quantale. Then,  $Q$ is a frame if and only if $(Q,\cdot,e)$ is a monoid with unital element $e=1$ and each element is idempotent with respect $\cdot$, here $1$ is the top of the complete lattice $(Q,\leq, \bigvee, 0, 1)$. 
\end{prop}   

Now the morphisms of quantales are $\bigvee$-morphisms such that respect the product. As in the case of idioms, we have the concept of inflator. In particular, we have that an inflator $d:Q\rightarrow Q$ on a quantale is  \emph{stable} if $ld(x)\leq d(lx)$ and $d(x)r\leq d(xr)$ for all $x,l,r\in Q.$ A \emph{pre-nucleus} in $Q$ is an inflator $d$ such that $d(x)d(y)\leq d(xy)$ holds for all $x,y \in Q.$ In this sense, a \emph{quantic nucleus} on $Q$ is an idempotent pre-nucleus $j$ on $Q.$ In general, the set of all nuclei $N(Q)$ is a complete lattice which is not a quantale, but the complete lattice structure ensures that there exists a unique quantic nucleus $\sigma\in N(Q)$ such that for each quantic nucleus $k\in N(Q)$, the set of fixed points of $k$, $Q_{k}$, that is a quotient of $Q$, is a frame if and only if $\sigma \leq k.$ The latter is proved in \cite[Lemma 3.2.5]{Ro}. There are more interesting families of inflators controlling the product and its comparison with  $\wedge$  on $Q$.

Some examples of quantales, as we mentioned before, are frames. Other examples come from ring theory. The lattice of left ideals, the one of right ideals and that of two sided ideals with the usual product, are examples  of quantales.  On the other hand, notice that $R-$fil the  lattice of all preradical filters with the operation of Gabriel multiplication of preradical filters satisfies (LQ), i.e. is a left quantale. See \cite[3.13 Proposition]{Golan}.

\section{Quasi-quantales}\label{sec:sec3}
\begin{dfn}\label{a}
	Let $A$ be a $\bigvee$-semilattice. We say that $A$ is a \emph{quasi-quantale} if it has an associative product $A\times{A}\rightarrow{A}$ such that for all directed subsets $X,Y\subseteq{A}$ and $a\in{A}$: 
	\[\left(\bigvee X\right)a=\bigvee\{xa\mid x\in X\},\leqno({\rm RDQ})\] and 
	\[a\left(\bigvee Y\right)=\bigvee\{ay\mid y\in Y\}.\leqno({\rm LDQ})\]
	We say $A$ is a \emph{left-unital  (resp. right-unital, resp. bilateral-unital)  quasi-quantale} if there exists $e\in{A}$ such that $e(a)=a$ (resp. $(a)e=a$, resp. $e(a)=a=(a)e$) for all $a\in{A}$. 
\end{dfn}
\begin{example}$\;$\label{exp1}
\begin{enumerate}
	\item If $A$ is an idiom, then $A$ is a bilateral-unital  quasi-quantale with $\wedge\colon A\times{A}\rightarrow{A}$.
	\item Let $R$ be a ring with identity and let $Bil(R)=\{I\subseteq{R}\mid I\text{ bilateral ideal }\}$. Then $Bil(R)$ is a bilateral-unital quasi-quantale with the product of ideals. Notice that $RI=I=IR$ for all $I\in Bil(R)$. Moreover, $Bil(R)$ satisfies
	\[\left(\sum_\EuScript{J}{I_i}\right)J=\sum_\EuScript{J}{(I_{i}J)}\]
	and
	\[J\left(\sum_\EuScript{J}{I_i}\right)=\sum_\EuScript{J}{(JI_i)},\]
for all $\{I_i\}_\EuScript{J}\subseteq{Bil(R)}$.

	Now, note that $\Lambda(R)=\{I\subseteq{R}\mid I\text{ is a left ideal }\}$ satisfies the latter distributive laws but it is just a left-unital  quasi-quantale.
	\item Let $R$-pr be the big lattice of preradicals on $\RMod,$ where the order is given as follows: $r,s\in\Rpr$ if and only if $r(M)\leq s(M)$ for all $M\in\RMod$. Their infimum $r\wedge s$  and supremum $r\vee s$ are defined as  $(r\wedge s)(M)=r(M)\cap s(M))$ and $(r\vee s)(M)=r(M)+ s(M))$  for every $M\in\RMod,$ respectively. In fact, by \cite[Theorem 8, 1.(b)]{pr} it follows that $R-$pr is a  bilateral-unital quasi-quantale  which is not a quantale with  $\wedge$ as product.  
	
On the other hand, given $r,s\in \Rpr$ consider  their product $rs$ which is defined as $(rs)(M)=r(s(M)),$ for every $M\in\RMod.$ Notice that  $\Rpr$ with this product just satisfies  {\rm (RQD)}
. See \cite[Theorem 8, 2.(b)]{pr}.
	\item Let $A$ be a lattice. The set of all inflators on $A$, denoted by $D(A)$   just satisfies (RQ) so in particular (RDQ). See  Lemma \ref{p2} (3) (a) and \cite{H3}.
	
	\item Let $M$ be a left $R$-module. 
	Given $N,L\in{\Lambda(M)}$, in \cite{B} was defined the product
	\[N_ML=\sum\{f(N)|f\in\Hom_{\R}(M,L)\}.\]
	In general, this product is not associative but if $M$ is projective in $\sigma[M]$ then it is. Therefore, if $M$ is projective in $\sigma[M]$, then $\Lambda(M)$ is a quasi-quantale. In fact, for all subset $\{N_i\}_\EuScript{J},$ 
	\[(\sum_{\EuScript{J}}{N_i})_{M}L=\sum_{\EuScript{J}}{({N_{i}}_{M}L)}.\]
	Let $\Lambda^{fi}(M)=\{N\leq{M}\mid N\text{ is fully invariant }\}$, that is, $f(N)\subseteq N$ for any endomorphism $f$ of $M$. Notice that $\Lambda^{fi}(M)$ is a right-unital quasi-quantale, because $N_MM=N$ for all $N\in{\Lambda^{fi}(M)}$. This example will be detailed in section \ref{sec4}.	
	
	\item 
	Any quantale is a quasi-quantale.
	\item The  lattice of all preradical filters, denoted by $R-$fil  with the operation of Gabriel multiplication of preradical filters  satisfies (LQ) and (RDQ), respectively. Therefore, $R-fil,$ is a quasi-quantale which is not a quantale.
See \cite[3.13 Proposition]{Golan}.

\end{enumerate}
\end{example}

\begin{prop}\label{003}
Let $A$ be a quasi-quantale and $x,y,z\in{A}$. Then, the following conditions hold.
\begin{enumerate}
\item If $x\leq y,$ then $zx\leq zy$ and $xz\leq yz.$
\item Moreover, if $1a\leq{a}$ for all $a\in{A}$, then
\begin{enumerate}[(a)]
	\item $xy\leq{y\wedge{x1}}$ 
	\item $x0=0$
\end{enumerate}
\end{enumerate}
\end{prop}
\begin{proof}
(1) Notice that $\{x,y\}\subseteq{A}$ is a directed subset, so $zy=z(x\vee y)=zx\vee zy.$
Thus $zx\leq{zy}$. Analogously, $yz=(x\vee y)z=xz\vee yz.$
So $xz\leq yz.$

\noindent (2) (a) We have that $\{y,1\}$ is directed, so $x1=x(y\vee{1})=xy\vee{x1}.$  Then $xy\leq{x1}$. On the other hand, $y\geq1y=(x\vee 1)y=xy \vee 1y=xy\vee y,$ then $xy\leq y$. Thus $xy\leq x1\wedge y.$

\noindent (2)(b)  By $(\textit{1}),$ it follows that $x0\leq{x1\wedge{0}}=0.$
\end{proof}

\begin{prop}\label{0032}
Let $A$ be a quasi-quantale.
The following conditions hold.
\begin{enumerate}
\item If $x\leq y$ and $z\leq v,$ then $xz\leq yv.$
\item $xy\vee xz\leq x(y\vee z)$ and   $yx\vee zx\leq (y\vee z)x,\,$ for all $x,y,z\in A.$
\item If $1a\leq{a}$ for all $a\in{A},$ then
\begin{enumerate}[(a)]
\item $xx:=x^2\leq x1$ and $x^2\leq x,$ for every $x\in A.$
\item  $(x\wedge y)^2\leq xy$ and $(x1\wedge y)^2\leq xy.$
\end{enumerate}
\end{enumerate}
\end{prop}
\begin{proof}
\noindent  (1) It follows from Proposition \ref{003} (2) (a).

\noindent  (2) Since $y\leq y\vee z,$ then  $xy\leq x(x\vee z).$ Similarly, $z\leq y\vee z$ implies that $xz\leq x(y\vee z).$ Therefore, $xy\vee xz\leq x(y\vee z).$

\noindent (3) (a)  It follows from Proposition \ref{003} (2) (a).

\noindent (3) (b) Since $x\wedge y \leq x$ and  $x\wedge y \leq y,$ by (2), we conclude that $(x\wedge y)^2\leq xy.$ On the other hand, since $x1\wedge y\leq x1$ and $x1\wedge y\leq y.$ By hypothesis,  $(1)$ and the associativity  of the product, we conclude that $(x1\wedge y)^2\leq (x1)y=x(1y)\leq xy.$
\end{proof}

\begin{prop}\label{0033}
Let $A$ be quasi-quantale. 
Consider the following statements. 
\begin{enumerate}
\item $xy=0$ if and only if $x1\wedge y=0,$ for every $x,y\in A.$
\item If $x^2=0,$ then $x=0,$  for every $x	\in A.$
\end{enumerate}

Then, the condition (1) always implies (2). If in addition, $A$ is a quasi-quantale which satisfies that $1a\leq{a}$ for all $a\in{A},$ the two conditions are equivalent.

\end{prop}
\begin{proof}
First, we prove that (1) implies (2).  Let $x\in A$ be such that $x^2=0.$ Then, by $(1)$ it follows that $x=x\wedge x=0.$

Now, suppose that $1a\leq{a}$ for all $a\in{A}.$ It remains to prove that (2) implies (1). Let $x,y\in A$ be such that $xy=0.$ Then, by Proposition \ref{0032} (3) (b), we have that $(x1\wedge y)^2\leq xy=0,$ so $(x1\wedge y)^2=0.$ Hence, $x1\wedge y=0.$
Conversely, suppose that $x1\wedge y=0.$ Then, it is immediately that $xy\leq x1\wedge y=0.$
\end{proof}

Recall that a lattice $A$ is {\em pseudomultiplicative} if it has a binary product
$(x,y)\mapsto xy$ which satisfies the following conditions for $x,y,z\in A$: \begin{enumerate}\item If $x\leq y$ then $xz\leq yz$ and $zx\leq zy$. \item $xy\leq x\wedge y$. \item $x(y\vee z)\leq xy\vee xz$.\end{enumerate} See \cite[Section 4]{FT} for more details.

From Proposition \ref{003}, it follows that a left-unital quasi-quantale (with $e=1$) generalizes the concept of a pseudomultiplicative lattice. Even though let $A$ be a bilateral quasi-quantale with $e=1,$ this one just will satisfy the above conditions (1) and (2). Indeed, a quasi-quantale, in general, does not hold the latter condition (3), see Proposition \ref{0032} (3) (b).


\begin{lemma}\label{03}
Let $A$ be a quasi-quantale which satisfies the following identities:
\[a(b\vee{c})=ab\vee{ac}\]
\[(b\vee{c})a=ba\vee{ca}\]
for all $a,b,c\in{A}$. Then $A$ is a quantale.
\end{lemma}

\begin{proof}
Let $X\subseteq{A}$. Define $Y=\{x_1\vee...\vee{x_n}\mid x_i\in{X}\}$. Then $X\subseteq{Y}$, so $\bigvee{X}\leq\bigvee{Y}$ and $Y$ is a direct subset of $A$. Since $A$ is quasi-quantale
\[a\left(\bigvee{Y}\right)=\bigvee\{ay\mid y\in{Y}\}.\]
Every $y\in{Y}$ is of the form $y=x_1\vee...\vee{x_n}$ with $x_i\in{X}$. Since $x_1\vee...\vee{x_n}\leq\bigvee{X}$ then 
\[\bigvee\{ay\mid y\in{Y}\}\leq{a\left(\bigvee{X}\right)}\leq{a\left(\bigvee{Y}\right),}\]
so
\[a\left(\bigvee{X}\right)=\bigvee\{ay\mid y\in{Y}\}.\]
We have that $X\subseteq{Y}$, whence $\bigvee\{ax\mid x\in{X}\}\leq\bigvee\{ay\mid y\in{Y}\}$. On the other hand, by hypothesis
\[ay=a(x_1\vee...\vee{x_n})=ax_1\vee...\vee{ax_n}\leq\bigvee\{ax\mid x\in{X}\},\]
thus $\bigvee\{ay\mid y\in{Y}\}\leq\bigvee\{ax\mid x\in{X}\}$. Hence
\[a\left(\bigvee{X}\right)=\bigvee\{ax\mid x\in{X}\}\]

\end{proof}

\begin{lemma}\label{04}
Let $(A,\leq,\vee,1)$ be a complete lattice. Then, $A$ is meet-continuous if and only if $A$ is a quasi-quantale such that
\begin{enumerate}[(a)]
	\item The binary operation in $A$ is commutative.
	\item A has an identity , $e=1$
\end{enumerate}
\end{lemma}
\begin{proof} First, suppose that $A$ is a meet-continuous complete lattice. It is clear that $A$ with the binary operation $\wedge$ is a quasi-quantale which satisfies (a) and (b).

Conversely, since $A$ is commutative, it follows that $1a=a=a1$. This implies that $ab\leq{a\wedge{b}}$. Thus $ab=a\wedge{b}.$ Therefore, $A$ is meet-continuous. 
\end{proof}

\begin{dfn}\label{b}
Let $A$ be a quasi-quantale. Consider $s\colon A\rightarrow{A}$ a inflator on $A.$ We say that:
\begin{enumerate}
	\item $s$ is \emph{contextual stable} if $s(a)x\leq{s(ax)}$ and $y(s(a))\leq{s(ya)}$ for all $a,x,y\in{A}$.
	\item $s$ is \emph{pre-multiplicative} if $s(a)\wedge{b}\leq{s(ab)}$ for all $a,b\in{A}$.
	\item $s$ is \emph{multiplicative} if $s(a)\wedge{s(b)}={s(ab)}$ for all $a,b\in{A}$.
	\item $s$ is a \emph{contextual pre-nucleus} if $s(a)s(b)\leq{s(ab)}$ for all $a,b\in{A}$. 
	\item $s$ is a \emph{contextual nucleus} if $s^2=s$ and $s$ is a contextual pre-nucleus.
\end{enumerate}
\end{dfn}

Let $A$ be a quasi-quantale and $j\colon A\rightarrow{A}$ a nucleus on $A$, we define a binary operation in $A_j$ given by $a,b\in{A_j}$, $a\cdot{b}=j(ab)$.

\begin{prop}\label{05}
Let $A$ be a quasi-quantale and $j$ be a contextual nucleus. Then $(A_j,\cdot)$ is a quasi-quantale. If $A$ is a left-unital quasi-quantale with identity  $e$ then $A_j$ is a left-unital quasi-quantale with identity  $j(e)$.
\end{prop}

\begin{proof}
Let $a,b,c\in{A_j}$. Consider the following inequalities,
\begin{equation}
abc\leq{j(ab)c}\leq{j(j(ab)c)}=j(ab)\cdot{c}=(a\cdot{b})\cdot{c}\end{equation}
\begin{equation}
j(ab)c=j(ab)j(c)\leq{j(abc)}
\end{equation}
Using (2), we have
\begin{equation}
(a\cdot{b})\cdot{c}=j(j(ab)c)\leq{j^2(abc)}=j(abc).
\end{equation}

By (1), $j(abc)\leq(a\cdot{b})\cdot{c},$ and by $(3)$, we have the equality
\[(a\cdot{b})\cdot{c}\leq{j(abc)}\leq(a\cdot{b})\cdot{c}\]

Analogously, $a\cdot(b\cdot{c})=j(abc)$. Then $a\cdot(b\cdot{c})=(a\cdot{b})\cdot{c}.\\$

Now, let $X\subseteq{A_j}$ be a directed subset and $a\in{A_j}$. Then, 
\[a\cdot\left(\bigvee^jX\right)=j\left(a\bigvee^jX\right)=j\left(aj\left(\bigvee{X}\right)\right)\geq{a\left(\bigvee{X}\right)}=\bigvee\{ax|x\in{X}\}.\]

Notice that $j(ax)\leq{j(\bigvee\{ax\mid x\in{X}\})}$ for all $x\in{X}$, so
\[j\left(a\cdot\left(\bigvee^jX\right)\right)\geq{j\left(\bigvee\{ax|x\in{X}\}\right)}\geq{\bigvee\{j(ax)|x\in{X}\}}.\]
Since $j$ is a nucleus,
\[a\cdot\left(\bigvee^jX\right)=j\left(a\cdot\left(\bigvee^jX\right)\right)\geq{j\left(\bigvee \{ j(ax)\mid x\in{X} \} \right)}=\bigvee^j\{a\cdot{x}\mid x\in{X}\}.\]
On the other hand,
\begin{displaymath}
\bigvee^j\{a\cdot{x}\mid x\in{X}\}=j\left(\bigvee\{j(ax)\mid x\in{X}\}\right)\geq{j\left(\bigvee\{ax\mid x\in{X}\}\right)}
\end{displaymath}
\begin{displaymath}
=j\left(a\bigvee{X}\right)
\geq{j(a)j\left(\bigvee{X}\right)}=a\cdot\bigvee^j{X}
\end{displaymath}
Thus
\[a\cdot\bigvee^j{X}\geq\bigvee^j\{ax|x\in{X}\}\geq{a\cdot\bigvee^j{X}}\]
Hence $(A_j,\cdot)$ is a quasi-quantale.

Now, suppose that $A$ is a left-unital quasi-quantale. Let $a\in{A_j}$, then 
\[a=ea\leq{j(e)a}\leq{j(j(e)a)}=j(j(e)j(a))\leq{j(ea)}=j(a)=a.\]
So $a=j(j(e)a)=j(e)\cdot{a}.$
Hence $A_j$ is a left-unital quasi-quantale.
\end{proof}

\begin{prop}\label{6}
Let $A$ be a quasi-quantale. Then, for each multiplicative nucleus $d$ on $A,$ the set $A_d$ is a meet-continuous lattice.
\end{prop}

\begin{proof}
Let $X\subseteq{A}_d$ be a directed subset and $a\in{A}_d$. Then,
\[d\left(\left(\bigvee{X}\right)\wedge{a}\right)=d\left(\left(\bigvee{X}\right)a\right)=d\left(\bigvee\{xa|x\in{X}\}\right)\leq{d\left(\bigvee\{d(xa)|x\in{X}\}\right)}\]
\[=d\left(\bigvee\{d(x)\wedge{d(a)}|x\in{X}\}\right)=d\left(\bigvee\{x\wedge{a}|x\in{X}\}\right)\leq{d\left(\left(\bigvee{X}\right)\wedge{a}\right)}.\]
Therefore, $d((\bigvee{X})\wedge{a})=d\left(\bigvee\{x\wedge{a}|x\in{X}\}\right)$. Thus
\[\left(\bigvee^d{X}\right)\wedge{a}=d\left(\bigvee{X}\right)\wedge{d(a)}=d\left(\left(\bigvee{X}\right)\wedge{a}\right)\]
\[=d\left(\bigvee\{x\wedge{a}|x\in{X}\}\right)=\bigvee^d{\{x\wedge{a}|x\in{X}\}}.\]
\end{proof}

\begin{cor}\label{05a}
Let $A$ be a quasi-quantale. Suppose that for any $X\subseteq{A}$ and $a\in{A}$ is satisfied 
\[\left(\bigvee{X}\right)a=\bigvee\{xa\mid x\in{X}\}.\]
Then, for each multiplicative nucleus $d$ on $A,$ the set $A_d$ is a frame.
\end{cor}

\begin{rem}
In Corollary \ref{05a}, if $A$ is an idiom then $A_j$ is an idiom.
\end{rem}

Now, consider $p:A\rightarrow{A}$ an inflator on $A$ such that
\[p(a)p(b)\leq{p(ab)}=p(a\wedge{b})=p(a)\wedge{p(b)}.\]
We call such $p$, \emph{idiomatic pre-nucleus}, and denote by $IP(A)$ the set of idiomatic pre-nucleus.

Let $S\subseteq{IP(A)}$ be a directed subset. Notice that

\[(\bigvee{S})(a)(\bigvee{S})(b)=(\bigvee\{s(a)\mid s\in{S}\})(\bigvee\{s(b)\mid s\in{S}\})\]
\[=\bigvee\{s(a)s'(b)\mid s,s'\in{S}\}.\]
Since $S$ is directed, for $s,s'\in{S}$ there exists $f\in{S}$ such that $s,s'\leq{f}$. Then
\[\bigvee\{s(a)s'(b)\mid s,s'\in{S}\}\leq\bigvee\{f(a)f(b)\mid f\in{S}\}\]
\[\leq\bigvee\{f(ab)\mid f\in{S}\}=\bigvee\{f(a\wedge{b})\mid f\in{S}\}.\]
Hence, $(\bigvee{S})(a)(\bigvee{S})(b)\leq(\bigvee{S})(ab)=(\bigvee{S})(a\wedge{b})$. Since the supremum of a directed set of pre-nucleus is a pre-nucleus, we obtain $\bigvee{S}\in{IP(A)}$.

Denote by $NI(A),$ the set of idiomatic nucleus.

\begin{prop}\label{06}
Let $A$ be a quasi-quantale. Then $NI(A)$ is a frame.
\end{prop}

\begin{proof}
Let $j\in{NI(A)}$ and $k$ a inflator. Let $G=\{f\mid f\wedge{k}\leq{j}\}$. Let $f,g\in{G}$ and $p=fg$. For all $x\in{A}$ we have that
\[k(x)\wedge{p(x)}=k(x)\wedge{f(g(x))}\leq{f(k(x))\wedge{f(g(x))}}\]
\[=f(k(x)\wedge{g(x)})\leq{f(j(x))}.\]
Then
\[k(x)\wedge{p(x)}\leq{f(j(x))}\wedge{k(x)}\leq{f(j(x))\wedge{k(j(x))}}\]
\[\leq(f\wedge{k})(j(x))\leq{j(j(x))}=j(x).\]
Thus, $G$ is directed and whence $\bigvee{G}$ is idiomatic.

Let $h=\bigvee{G}$. Consider
\[j\circ{(h\wedge{k})}(x)=j(h(x)\wedge{k(x)})=j(h(x)k(x))=j\left(\left(\bigvee{G}\right)(x)k(x)\right)\]
\[=j\left(\bigvee\{g(x)k(x)\mid g\in{G}\}\right)\leq{j\left(\bigvee\{j(g(x)k(x))\mid g\in{G}\}\right)}\]
\[\leq{j\left( \bigvee\{j(g(x)\wedge{k(x)})\mid g\in{G}\}\right)} \leq{j(j(j(x)))}=j(x)\]
Thus, $j\circ{(h\wedge{k})}=j$. Then $h\wedge{k}\leq{j}$, that is, $h\in{G}$. This implies that $h^2=h,$ and so $h\in{NI(A)}$. 

Hence, $NI(A)$ has implication. Consequently,  $NI(A)$ is a frame.
\end{proof} 

\begin{dfn}\label{c}
Let $A$ be a quasi-quantale. An element $1\neq{p}\in{A}$ is {\em prime} if whenever $ab\leq{p}$ then $a\leq{p}$ or $b\leq{p}$.
\end{dfn}

\begin{dfn}\label{d} Let $A$ be a quasi-quantale and $B$ a sub $\bigvee$-semilattice.  We say that $B$ is a {\em subquasi-quantale of $A$} if 
\[\left(\bigvee{X}\right)a=\bigvee\{xa\mid x\in{X}\}\] and
\[a\left(\bigvee{Y}\right)=\bigvee\{ay\mid y\in{Y}\},\]
for all directed subsets $X,Y\subseteq{B}$ and $a\in{B}$.
\end{dfn}

\begin{dfn}\label{e}
Let $A$ be a quasi-quantale and $B$ a subquasi-quantale of $A$. An element $1\neq{p}\in{A}$ is a {\em prime element relative to $B$} if whenever $ab\leq{p}$ with $a,b\in{B}$ then $a\leq{p}$ or $b\leq{p}$.
\end{dfn}

It is clear that a prime element in $A$ is a prime element relative to $A$. 

We define the spectrum relative to $B$ of $A$ as
\[Spec_B(A)=\{p\in{A}\mid p\;is\;prime\;relative\;to\;B\}\]

If $B=A$, the we just write $Spec(A)$.

\begin{lemma}\label{sc}
Let $B$ be a subquasi-quantale of a quasi-quantale $A$. Suppose that $0,1\in{B}$ and $1b,b1\leq{b}$ for all $b\in{B}$, then
\begin{enumerate}
	\item $ab\leq{a\wedge{b}}$ for all $a,b\in{B}.$
	\item Let $p\in{A}$ a prime element relative to $B$. If $a,b\in{B}$, then $ab\leq{p}$ if and only if $a\leq{p}$ or $b\leq{p}$.
\end{enumerate}
\end{lemma}

\begin{proof}
$(1)$. It follows by Proposition \ref{003}.$(2)$.$(a)$ that $ab\leq{a\wedge{1b}}$. By hypothesis $1b\leq{b}$, so $ab\leq{a\wedge{b}}$.

$(2)$. Let $a,b\in{B}$ and $p\in{A}$ a prime element relative to $B$. 
First, if $ab\leq p,$ 
then, by definition, it follows that $a\leq p$ or $b\leq p.$

Conversely, suppose that $a\leq{p}$. By $(1),$ $ab\leq{a\wedge{b}}\leq{a}$, then $ab\leq{p}$. Analogously if $b\leq{p}$.
\end{proof}

From now on, let $B$ be a subquasi-quantale of a quantale $A$ and suppose that $0,1\in{B}$ and $1b,b1\leq{b}$ for all $b\in{B}$.

\begin{prop}\label{t}
Let $B$ be a subquasi-quantale of a quasi-quantale $A$. Then $Spec_B(A)$ is a topological space, where the closed subsets are subsets given by
\[\mathcal{V}(b)=\{p\in{Spec_B(A)}\mid b\leq{p}\},\]
with $b\in{B}$. 

In dual form, the open subsets are of the form
\[\mathcal{U}(b)=\{p\in{Spec_B(A)}\mid b\nleq{p}\},\]
with $b\in{B}$.
\end{prop}

\begin{proof}
It clear that $\mathcal{U}(0)=\emptyset$ and $\mathcal{U}(1)=Spec_B(A)$. Let $\{b_i\}_I$ be a family of elements in $B$. Then
\[\mathcal{U}\left(\bigvee_I{b_i}\right)=\left\{p\in{Spec_B(A)}\mid\bigvee_I{b_i}\nleq{p}\right\}=\bigcup_I\{p\in{Spec_B(A)}\mid b_i\nleq{p}\}\]
\[=\bigcup_I{\mathcal{U}(b_i)}\]
Now, let $a,b\in{B}$. Then, by Lemma \ref{sc}.$(2),$
\[\mathcal{U}(ab)=\{p\in{Spec_B(A)}\mid ab\nleq{p}\}=\{p\in{Spec_B(A)}\mid a\nleq{p}\;and\;b\nleq{p}\}\]
\[=(\{p\in{Spec_B(A)}\mid a\nleq{p}\})\cap(\{p'\in{Spec_B(A)}\mid b\nleq{p'}\})=\mathcal{U}(a)\cap\mathcal{U}(b).\]
\end{proof}

\begin{rem}
Let $\mathcal{O}(Spec_B(A))$ be the frame of open subsets of $Spec_B(A)$. We have an adjunction of $\bigvee$-morphisms 
\[\xymatrix@=30mm{B\ \ \ar@/^/[r]^{\mathcal{U}} & \mathcal{O}(Spec_B(A))\ \ \ar@/^/[l]^{\mathcal{U}_*}}\]
Where $\mathcal{U}_*$ is defined as
\[\mathcal{U}_*(W)=\bigvee\{b\in{B}\mid\mathcal{U}(b)\subseteq{W}\}\]

The composition $\mu=\mathcal{U}_*\circ\mathcal{U}$ is a closure operator in $B$.
\end{rem}

\begin{prop}\label{41}
Let $b\in{B}$. Then $\mu(b)$ is the largest element in $B$ such that $\mu(b)\leq\bigwedge\{p\in{Spec_B(A)}\mid p\in\mathcal{V}(b)\}.$
\end{prop}

\begin{proof}
By definition, 
\[\mu(b)=\mathcal{U}_*(\mathcal{U}(b))=\bigvee\{c\in{B}\mid \mathcal{U}(c)\subseteq\mathcal{U}(b)\}\]
\[=\bigvee\{c\in{B}\mid\mathcal{V}(b)\subseteq\mathcal{V}(c)\}\leq\bigwedge\{p\in{Spec_B(A)}\mid p\in\mathcal{V}(b)\}.\]
Let $x\in{B}$ such that $x\leq{p}$ for all $p\in\mathcal{V}(b)$, then $\mathcal{V}(b)\subseteq\mathcal{V}(x).$ Thus, $x\leq\bigvee\{c\in{A}\mid\mathcal{V}(b)\subseteq\mathcal{V}(c)\},$ whence $x\leq\mu(b)$.
\end{proof}

\begin{thm}\label{3.15}
The closure operator in $\mu\colon B\rightarrow{B}$ is a multiplicative pre-nucleus. 
\end{thm}

\begin{proof}
Let $a,b\in{B}$. By Lemma \ref{sc}.$(1)$, $ab\leq{a\wedge{b}}$. Thus $\mu(ab)\leq\mu(a)\wedge\mu(b)$.

By Proposition \ref{41},
\begin{displaymath}
\mu(a)\wedge\mu(b){\leq}\left(\bigwedge\{q\in{Spec_B(A)}{\mid} q{\in}\mathcal{V}(a)\}\right)\wedge\left(\bigwedge\{q'\in{Spec_B(A)}\mid q'{\in}\mathcal{V}(b)\}\right).
\end{displaymath}

Let $p\in{Spec_B(A)}$ such that $ab\leq{p}$, then $a\leq{p}$ or $b\leq{p}$. If $a\leq{p}$ then, 
\[\bigwedge\{q\in{Spec_B(A)}\mid q\in\mathcal{V}(a)\}\}\leq{p}.\]
Thus, $\mu(a)\wedge\mu(b)\leq{p}$. 
Analogously, if $b\leq{p}$ then $\mu(a)\wedge\mu(b)\leq{p}$. Hence
\[\mu(a)\wedge\mu(b)\leq\bigwedge\{p\in{Spec_B(A)}\mid p\in\mathcal{V}(ab)\}\]
By Proposition \ref{41}, $\mu(ab)$ is the largest element in $B$ less or equal than $\bigwedge\{p\in{Spec_B(A)}\mid p\in\mathcal{V}(ab)\}$, therefore $\mu(a)\wedge\mu(b)\leq\mu(ab)$. Thus $\mu$ is multiplicative. 

Now, since $ab\leq{a\wedge{b}}$ then $\mu(a)\wedge\mu(b)=\mu(ab)\leq\mu(a\wedge{b})$. The other inequality always holds. Thus $\mu$ is a pre-nucleus.
\end{proof}

\begin{cor}
$A_\mu$ is an meet-continuous lattice.
\end{cor}

\begin{proof}
It follows by Proposition \ref{6}.
\end{proof}

\section{The Large Spectrum}\label{sec4}

In this section, we introduce a new spectrum for a module and we give some characterizations of modules with this space.


Firstly, recall that for any left $R$-module $M,$ in \cite[Lemma 2.1]{B}  was defined the product  of submodules $N,L\in{\Lambda(M)},$
	\[N_ML:=\sum\{f(N)|f\in\Hom_{\R}(M,L)\}.\]
	
\begin{rem}\label{tpro} \cite[ Proposition 1.3]{Lg}
\; This product satisfies the following properties for every submodules $K,K'\in\Lambda(M)$:
\begin{enumerate}[(1)]
\item If $K\subseteq K'$ then $K_{M}X\subseteq K'_{M}X$ for every module $X$.
\item If $X$ is a left module and $Y\subseteq X$ then $K_{M}Y\subseteq K_{M}X$.
\item $M_{M}X\subseteq X$ for every module $X$.
\item $K_{M}X=0$ if and only if $f(K)=0$ for all $f\in \Hom(M,X)$.
\item $0_{M}X=0$, for every module $X$.
\item Let $\{X_{i}\mid i\in I\}$ be a family of submodules of $M$ then $\sum_{i\in I}\left(K_MX_i\right)\subseteq K_{M}\left(\sum_{i\in I}X_i\right)$.
\item $\left(\sum_{i\in I}K_{i}\right)_{M}N=\sum_{i\in I}K_{iM}N$ for every family of submodules $\{K_{i}\mid i\in I\}$ of $M$.

\end{enumerate}

\end{rem}
\begin{proof}
The proof of this is in \cite[ Proposition 1.3]{Lg}. Here we only give the proof of $(6)$ and $(7)$.

$(6)$ Let $\{X_{i}\mid i\in I\}$ be a family of submodules of $M$. Since $X_i\leq\sum_{i\in I}X_i$, by (2) of this proposition $\left(K_MX_i\right)\subseteq K_{M}\left(\sum_{i\in I}X_i\right)$. Thus,
\[\sum_{i\in I}\left(K_MX_i\right)\subseteq K_{M}\left(\sum_{i\in I}X_i\right).\]

$(7)$ Let $\{K_{i}\mid i\in I\}$ be a family the submodules of $M$. Then
\[\left(\sum_{i\in I}K_{i}\right) _{M}N=\sum\left\{\left(\sum_{i\in I}K_i\right)|f\in\Hom_{\R}(M,N)\right\}\]
\[=\sum\left\{\sum_{i\in I}f\left(K_i\right)|f\in\Hom_{\R}(M,N)\right\}=\sum_{i\in I}\sum\left\{f(K_i)|f\in\Hom_{\R}(M,N)\right\}\]
\[=\sum_{i\in I}\left({K_i}_MN\right).\]
\end{proof}

Now, recall that an $R-$module $N$ is said to be subgenerated by $M$ if $N$ is isomorphic to a submodule of an $M-$generated module. It is denoted by $\sigma[M]$ the full subcategory of $\RMod$ whose objects are all $R-$modules subgenerated by $M.$   Also,  a  module $N$ is called a subgenerator in $\sigma[M]$ if $\sigma[ M]=\sigma[N]$. In particular, $M$ is a subgenerator in $\RMod$ if $\sigma[M]=\RMod.$	See \cite[\S 15]{W} for more details.
	
Remember that $N\leq M$ is a fully invariant submodule of $M$ if $f(N)\leq N$ for all $f$ endomorphism of $M$. The set of fully invariant submodules of $M$ is denoted by $\Lambda^{fi}(M)$.

\begin{rem}\label{tpro1}
Let $M\in \RMod$ and projective in $\sigma[M]$.  The following conditions hold.
\begin{enumerate}[(1)]
\item The product $-_{M}-:\Lambda(M)\times \Lambda(M)\to \Lambda(M)$ is associative. 
\item $K_{M}\left(\sum_{i\in I}X_{i}\right)=\sum_{i}\left(K_{M}X_{i}\right)$ for every directed family of submodules $\{X_{i}\mid i\in I\}$ of $M$.
\item If $N,L\in \Lambda^{fi}(M),$ then  $N_{M}L\in  \Lambda^{fi}(M)$ and hence the product $-_{M}-$ is well restricted in $\Lambda^{fi}(M).$
\end{enumerate}
\end{rem}
\begin{proof}
$(1)$ It follows by \cite[Proposition 5.6]{Bc}.

$(2)$ Let $\{X_i|i\in I\}$ be a directed family of submodules of $M$. Let $\sum_{j\in J} f_j(k_j)\in K_M\left(\sum_{i\in I}X_i\right)$, with $f_j:M\to \sum_{i\in I}X_i$. Since $M$ is projective in $\sigma[M]$ for each $f_j$ there exists $g_{i_j}:M\to X_i$ such that $\sum_{i\in I}g_{i_j}(k_j)=f_j(k_j).$
\[\xymatrix{ & M\ar[d]^{f_j}\ar[dl]_{\bigoplus g_{i_j}} \\ \bigoplus_{i\in I}X_i\ar@{-{>>}}[r] & \sum_{i\in I}X_i}\]

Then 
\[\sum_{j\in J}f_j(k_j)=\sum_{j\in{J}}\sum_{i\in I}g_{i_j}(k_j)\in\sum{X_{i_j}}.\]
Since this sum is finite and $\{X_i|i\in I\}$ is directed, there exists $l\in I$ such that $\sum{X_{i_j}}\subseteq{X_l}$. Thus
\[\sum_{j\in J}f_j(k_j)=\sum_{j\in{J}}\sum_{i\in I}g_{i_j}(k_j)\in K_MX_l\subseteq\sum_{i\in I}\left(K_{M}X_{i}\right)\]

The other contention follows by \ref{tpro}.(6).

$(3)$ Let $N,L\in\Lambda^{fi}(M)$ and $g:M\to M$. Then 
\[g(N_ML)=g\left(\sum\{f(N)|f\in\Hom_{\R}(M,L)\}\right)=\sum\{gf(N)|f\in\Hom_{\R}(M,L)\}\]
Since $L\in\Lambda^{fi}(M)$, $gf\in\Hom_{\R}(M,L)$. Hence
\[g(N_ML)\subseteq N_ML.\]
\end{proof}

Notice that in Remark \ref{tpro1} (1),  the projectivity condition for $M$ is necessary as  is shown in the following example which was taken from \cite[Lemma  2.1 (vi)]{B}:   consider $\mathbb{Z}$ the additive group of integers and $\mathbb{Q}$ the additive group of rational numbers, then $0=(\mathbb{Z\,}_{\,\mathbb{Q}}\,\mathbb{Z}\,)_{\mathbb{Q}}\, \mathbb{Q}\neq \mathbb{Z\,}_{\mathbb{Q}} \,(\mathbb{Z\,} _{\,\mathbb{Q}\,} \mathbb{Q} )=\mathbb{Q}.$

\begin{prop}\label{pro1} Let $M\in \RMod.$ If $M$ is projective in $\sigma[M],$ the following conditions hold.
\begin{enumerate}[(1)]
\item $\Lambda(M)$ is a quasi-quantale. In fact, for all subset $\{N_i\}_\EuScript{J},$ 
	$(\sum_{\EuScript{J}}{N_i})_{M}L=\sum_{\EuScript{J}}{({N_{i}}_{M}L)}.$

\item $\Lambda^{fi}(M)$ is a right-unital subquasi-quantale of $\Lambda(M)$.
\end{enumerate}
\end{prop}	

\begin{proof} (1) By Remark \ref{tpro1} (1), the product $-_{M}-:\Lambda(M)\times \Lambda(M)\to \Lambda(M)$ is associative, and by Remark \ref{tpro} (7) and Remark \ref{tpro1} (2), $\Lambda(M)$ is a quasi-quantale.

(2) Since the product $-_{M}-$ is well restricted in $\Lambda^{fi}(M)$, then $\Lambda^{fi}(M)$ is a subquasi-quantale. Now, if $N\in\Lambda^{fi}(M)$ then
\[N_MM=\sum\{f(N)|f\in\Hom_{\R}(M,M)\}\subseteq{N},\]
but $N=Id_M(N)\leq{N_MM}$. Hence $N_MM=N$. Thus $\Lambda^{fi}(M)$ is a right-unital quasi-quantale.
\end{proof}

In this study, we also have preradicals theory. Recall that a \textit{ preradical} $r$ on $\RMod$ is a subfunctor of the identity functor  on $\RMod,$ i.e. $r$ assigns to each $M$  a submodule $r(M)$ in such a way that every $f\in\Hom_{\R}(M,N)$ restricts to  a homomorphism $r(f)\in \Hom_{\R}(r(M),r(N)).$  In particular, $r(\R)$ is a two-sided ideal.
The big lattice of  preradicals in $\RMod$   is denoted by $\Rpr.$ In fact, $\Rpr$ is a left -unital quasi-quantale with the product  $\wedge$, as was mentioned in  Example \ref{exp1} (3).

Let $M\in \RMod.$ For each $N\in\Lambda^{fi}(M),$  there are two distinguished preradicals, $\alpha^M_N$ and $\omega^M_N,$ which are defined as follow

\[\alpha^{M}_{N}(L){:=}\sum\{f(N) \mid  f\in\Hom_{\R}(M,L)\}\] 
and  
\[\omega^{M}_{N}(L){:=}\bigcap\{f^{-1}(N) \mid f\in\Hom_{\R}(L,M)\},\] 

\noindent for each $L\in \RMod.$  

\begin{rem} \cite[Proposition 5]{pr}. 
If $N$ is a fully invariant submodule of $M,$ then the following conditions are satisfied.
\begin{enumerate}[(1)] 
\item The preradicals $\alpha^M_N$ and $\omega^M_N$ have the property that $\alpha^M_N(M)=N$ and $\omega^M_N(M)=N$ respectively. 
\item The class $\{r\in \Rpr \mid r(M)=N\}$ is precisely the interval $[\alpha^{M}_{N},\omega^{M}_{N}]$.
\end{enumerate}
\end{rem}

 The reader can find more properties of these preradicals in \cite{pr}, \cite{PI}, \cite{S} and 
 \cite[Proposition 1.3]{BPI}.

\begin{dfn}\label{f} 
Let $M,\, P\in\RMod$ such that $P\leq M.$ Define
 \[\eta^M_P\colon\RMod\to \RMod\]
 
 \[\eta^M_P(L):=\bigcap \{f^{-1}(P)\in\RMod | f\in\Hom_{\R}(L,M) \},\]
for each $L\in\RMod.$
\end{dfn} 

It is clear that $\eta^M_P$ is a preradical. Also, it is clear that $\eta^M_P(M)\leq P.$
When $P\in\Lambda^{fi}(M),$ it follows that $P=\eta^M_P(M).$ And in this case, $\omega^M_P=\eta^M_P$.


\begin{rem} \label{rk 6.2}
\begin{enumerate}
\item $\eta^M_P\preceq \omega^M_{\eta^ M_{P} (M)}$
\item If  $N,L\in \Lambda^{fi}(M)$ and $N_ML\leq P,$ then $\omega^M_{(N_ML)}\preceq \eta^M_P.$
\end{enumerate}
\end{rem}

\noindent 

\begin{rem}  \label{rk 6.3}
Let $M\in\RMod.$
\begin{enumerate}
	\item If $P,Q\leq M$ satisfy that $P\leq Q,$ then $\eta^M_P\preceq \eta^M_Q.$
	\item If $r\in\Rpr$  and $P\leq M$ satisfy that $r(M)\leq \eta^M_P(M), $
then \[r\preceq \omega^M_{r(M)}\preceq \eta^M_P \preceq \omega^M_{\eta^ M_{P} (M)}.\]
\end{enumerate}
\end{rem}

\begin{rem}
Let $M$ be an $R$-module. We can consider $Spec(\Lambda^{fi}(M))$. Notice that the prime submodules defined in \cite{PI} are the elements in $Spec(\Lambda^{fi}(M))$.
\end{rem}

\begin{prop} \label{pro 6.4} Let $M\in\RMod$ and $P\leq M.$ If $\eta^M_P$ is a prime preradical i.e. $\eta^M_P\in{Spec(\Rpr)}$ and $\eta^M_P(M)\neq M,$ then $\eta^M_P(M)\in{Spec(\Lambda^{fi}(M))}$. Moreover, 
$\eta^M_P(M)$ is the largest element in $Spec(\Lambda^{fi}(M))$ which is contained in $P.$
\end{prop}
\begin{proof} Let $N,L\in\Lambda^{fi}(M)$ such that $N_ML\leq \eta^M_P(M).$ This is, $(\alpha^M_N\cdot \alpha^M_L)(M)\leq \eta^M_P(M).$ By Remark \ref{rk 6.3}(1)
 it follows that $\alpha^M_N\cdot \alpha^M_L\preceq \eta^M_P. $ Since $\eta^M_P$ is a prime preradical, we get $\alpha^M_N\preceq \eta^M_P$ or $\alpha^M_L\preceq \eta^M_P.$ Thus, $N\leq \eta^M_P(M)$ or  $L\leq \eta^M_P(M).$ Therefore, $\eta^M_P(M)\in{Spec(\Lambda^{fi}(M))}$.
 
 Finally, let $K\in\Lambda^{fi}(M)$ be a prime submodule of $M$ such that $K\leq P.$ Then, $\omega^M_K\preceq \eta^M_P.$ Thus, applying $M$ in the last inequality we obtain $K=\omega^M_K(M)\leq\eta^M_P(M)\leq P.$
\end{proof}

From now on, $M$ will be assumed projective en $\sigma[M]$, in order to have that $\Lambda^{fi}(M)$ is a right-unital subquasi-quantale of $\Lambda(M)$.

Let $M$ be an $R$-module. If $P$ is a prime element of $\Lambda(M)$ relative to $\Lambda^{fi}(M)$, then we will call it {\em a large prime submodule of $M$}. We will write $LgSpec(M)$ instead of $Spec_{\Lambda^{fi}(M)}(\Lambda(M))$.

\begin{prop}\label{pro 6.5} Let $M,\, P\in\RMod$  such that $P\leq M.$
Then, the following conditions are equivalent.
\begin{enumerate}
\item $P$ is a large prime submodule of $M.$
\item $\eta^M_P$ is a prime preradical.
\end{enumerate}
\end{prop}
\begin{proof}
(1)$\Rightarrow$(2)  From the hypothesis, it is clear that $P\neq M.$ So, $\eta^M_P\neq \I.$  Let $r,s\in\Rpr$ such that $r\cdot s\preceq \eta^M_P.$ Evaluating this in $M$ and  using \cite[Proposition 14 (2)]{PI}, we obtain $r(M)_Ms(M)\leq\eta^M_P(M).$  Thus, $r(M)_Ms(M)\leq\eta^M_P(M)\leq P.$ so, by hypothesis, it follows that $r(M)\leq P$ or $s(M)\leq P.$ And by Proposition \ref{pro 6.4}, $\eta^M_P(M)$ is the largest fully invariant submodule of $M$ which is contained in $P.$ Hence, $r(M)\leq\eta^M_P(M)$ or $s(M)\leq\eta^M_P(M).$ Consequently, $r\preceq \eta^M_P$ or $s\preceq \eta^M_P.\\$

\noindent (2)$\Rightarrow$ (1)  By hypothesis, $\eta^M_P$ is a prime preradical, so in particular $P\lneqq M.$ Let $N,L\in \Lambda^{fi}(M)$ such that $N_ML=\alpha^M_N(L)\leq P.$ Notice that $N_ML=\alpha^M_N(L)=(\alpha^M_N\cdot \alpha^M_L)( M).$ 
Then, $\alpha^M_N\cdot \alpha^M_L\preceq \omega^M_{N_ML}.$ Thus, by Remark \ref{rk 6.3} (2) we get $\alpha^M_N\cdot \alpha^M_L\preceq  \eta^M_P.$ Since $ \eta^M_P$ is prime, then  $\alpha^M_N\preceq \eta^M_P$ or $\alpha^M_L\preceq \eta^M_P.$
Applying $M$ we obtain $N=\alpha^M_N(M)\leq \eta^M_P(M)\leq P$ or $L=\alpha^M_L(M)\leq \eta^M_P(M)\leq P.$ By Remark \ref{rk 6.3} (2), we conclude that $r\preceq \eta^M_P$ or $s\preceq \eta^M_P.$ Hence, $\eta^M_P$ is a prime preradical.
\end{proof}

\begin{cor}\label{cor 6.6} Let $M\in\RMod$ and $P\leq M.$ If $P$ is a large prime submodule of $M,$ then $\eta^M_P(M)$ is the largest element in $Spec(\Lambda^{fi}(M))$ which is contained in $P.$ 
\end{cor}
\begin{proof} It follows from Propositions \ref{pro 6.4} and \ref{pro 6.5}.
\end{proof}

\begin{dfn}\label{g}
Following \cite{PI}, an $R$-module $M$ is a \emph{prime} module if 
\[0\in{Spec(\Lambda^{fi}(M))}.\]
\end{dfn}

\begin{cor} \label{1}
Let $M\in\RMod$ and $P\leq M.$ Then, the following conditions hold.
\begin{enumerate}

\item If $P$ is a large prime submodule of $M,$  then $M/\eta^M_P(M)$ is a prime module.

\item If $M$ is a quasi-projective module and $M/P$ is a prime module, then $P$ is a large prime submodule.
\end{enumerate}
\end{cor}
\begin{proof}
(1) It is a consequence from Corollary \ref{cor 6.6} and \cite[Proposition 18.1]{PI}.

(2) Note that in \cite[Proposition 18.2]{PI} is not necessary the hypothesis of $P\in\Lambda^{fi}(M)$.
\end{proof}

\begin{example}
If $\mathcal{M}<{M}$ is a maximal submodule then $M/\mathcal{M}$ is a simple module, hence it is a prime module. Therefore, every maximal submodule of $M$ is a large prime submodule.
\end{example}

\begin{dfn}\label{h}
An $R$-module $M$ is \emph{ FI-simple } if the only fully invariant submodules of $M$ are $0$ and $M$.
\end{dfn}

\begin{prop}
If $M$ is an FI-simple module, then $M$ is cogenerated by all its non-zero factors modules.
\end{prop}

\begin{proof}
If $M$ is FI-simple, it is clear that $LgSpec(M)=\{N\leq{M}\mid N\neq{M}\}$ and $Spec(\Lambda^{fi}(M))={0}$. By corollary \ref{cor 6.6}, $\eta^M_N(M)=0$ for all $N\leq{M},$ with $N\neq{M}$. This implies that there exists a monomorphism $M\rightarrow{(M/N)^S}$, where $S=End_R(M)$ for all proper submodule of $M$.  
\end{proof}

\begin{rem}
In \cite{Z} the author writes $Spec^{fp}(M)$ instead of $Spec(\Lambda^{fi}(M))$. If $M$ is projective in $\sigma[M]$ then $LgSpec(M)$ is a topological space by Proposition \ref{t} and Corollary \ref{1}. We can see that $Spec(\Lambda^{fi}(M))$ is a subspace of $LgSpec(M)$. Recall that an $\R-$module is \textit{duo} if $\Lambda(M)=\Lambda^{fi}(M).$ Note that if $M$ is a duo module then $Spec(\Lambda^{fi}(M))=LgSpec(M)=Spec(\Lambda(M))$. We will call $LgSpec(M)$ the \emph{large spectrum }of $M$.
\end{rem}

\begin{dfn}\label{i}
A module $M$ is \emph{coatomic} if every submodule is contained in a maximal submodule of $M$.
\end{dfn}

\begin{example} $\,$
\begin{enumerate}
	\item Every finitely generated and every semisimple module is coatomic.
	\item Semiperfect modules are coatomic.
	\item If $R$ is a left perfect ring every left $R$-module is coatomic. (see \cite{Gun})
\end{enumerate}
\end{example}

Note that in general, it could be that $LgSpec(M)=\emptyset$. If we assume that $M$ is a coatomic module then $LgSpec(M)$ is not the empty set.  

We will denote $Spec(\Lambda(M))$ just by $Spec(M)$.

\begin{prop}
Let $M$ be quasi-projective. Then $Spec(\Lambda^{fi}(M))$ is a dense subspace of $LgSpec(M)$
\end{prop}

\begin{proof}
Let  $\mathcal{U}(N)\neq\emptyset$ be an open set of $LgSpec(M)$ such that 
\[\mathcal{U}(N)\cap Spec(\Lambda^{fi}(M))=\emptyset.\]
 Thus, the elements of $\mathcal{U}(N)$ are not fully invariant. Let $P\in\mathcal{U}(N)$. By Corollary \ref{cor 6.6}, there exists $Q\in{Spec(\Lambda^{fi}(M))}$ such that $Q\subseteq{P}$, which is a contradiction. Thus $\mathcal{U}(N)=\emptyset$. So, $Spec(\Lambda^{fi}(M))$ is dense in $LgSpec(M)$.
\end{proof}

Let $\mathcal{O}(LgSpec(M))$ be the frame of open subsets of $LgSpec(M)$. Then we have a morphism of $\bigvee$-semilattices
\[\mathcal{U}\colon\Lambda^{fi}(M)\rightarrow\mathcal{O}(LgSpec(M))\]
given by $\mathcal{U}(N)=\{P\in{LgSpec(M)}\mid N\nsubseteq{P}\}$. 

This morphism has a right adjunct $\mathcal{U}_*\colon\mathcal{O}(LgSpec(M))\rightarrow{\Lambda^{fi}(M)}$ given by
\[\mathcal{U}_*(A)=\sum\{K\in{\Lambda^{fi}(M)}\mid\mathcal{U}(K)\subseteq{A}\}.\]

\begin{prop}\label{2} 
Let $N\in{\Lambda^{fi}(M)}$. Then $(\mathcal{U}_*\circ\mathcal{U})(N)$ is the largest fully invariant submodule of $M$ contained in $\displaystyle\bigcap_{P\in\mathcal{V}(N)}{P}.$
\end{prop}

\begin{proof}
It follows by Proposition \ref{41}.
\end{proof}

\begin{dfn}\label{sp}
Let $M$ be an $R$-module and $M\neq{N}\in\Lambda^{fi}(M)$. We say that $N$ is {\em semiprime in $M$} if whenever $K_MK\leq{N}$ with $K\in\Lambda^{fi}(M)$ then $K\leq{N}$. We say that $M$ is a {\em semiprime} module if $0$ is semiprime.
\end{dfn}

\begin{rem}\label{4.23}
The last definition is given in \cite{S}. Note that if $\{P_i\}_I$ is a family of large primes submodules of $M$, then $\displaystyle\bigcap_I{P_i}$ is not necessary a fully invariant submodule of $M$ but $\displaystyle\bigcap_I{P_i}$ satisfies the property that for every $L\in\Lambda^{fi}(M)$ such that $\displaystyle L_ML\leq\bigcap_I{P_i}$ then $\displaystyle L\leq\bigcap_I{P_i}$.
\end{rem}

\begin{prop}\label{3}
Let $\mu=\mathcal{U}_*\circ\mathcal{U}$ be as in Proposition \ref{2}. Let $N\in\Lambda^{fi}(M)$ then, $\mu(N)=N$ if and only if $N$ is semiprime in $M$ or $N=M$. 
\end{prop}

\begin{proof}
Assume that $\mu(N)=N$ and let $L\in\Lambda^{fi}(M)$ such that $L_{M}L\leq{N}.$ We have that $\mu(N)\leq\displaystyle{\bigcap_{P\in\mathcal{V}(N)}}{P}$, so $\displaystyle L_ML\leq\bigcap_{P\in\mathcal{V}(N)}P$. By Remark \ref{4.23}, $\displaystyle L\leq\bigcap_{P\in\mathcal{V}(N)}P$. Since $L$ is a fully invariant submodule of $M$, by Proposition \ref{2} $L\leq\mu(N)$. Thus $\mu(N)$ is semiprime in $M$. 

Now suppose that $N$ is semiprime in $M$. By \cite[Proposition 1.12]{Gold}, \[N=\bigcap\{Q\mid N\leq{Q}\;Q\in{Spec(\Lambda^{fi}(M))}\}.\]
 Since $Spec(\Lambda^{fi}(M)){\subseteq}{LgSpec(M)},$
  then 
  \[\bigcap_{P\in\mathcal{V}(N)}P\subseteq\bigcap\{Q\in{Spec(\Lambda^{fi}(M))}\mid N\leq{Q} \}=N.\] This implies that $N=\mu(N)$.
\end{proof}



\begin{cor}\label{42}
The closure operator $\mu\colon\Lambda^{fi}(M)\rightarrow\Lambda^{fi}(M)$ is a multiplicative pre-nucleus.
\end{cor}

\begin{proof}
By Theorem \ref{3.15}.
\end{proof}

\begin{rem}\label{5}
By definition $\mu$ is a closure operator and by Corollary \ref{42} $\mu$ is a pre-nucleus then $\mu$ is a nucleus in $\Lambda^{fi}(M).$
\end{rem}

\begin{prop}\label{7}
Let $M\in{\RMod}$ and 
\[SP(M)=\{N\in{\Lambda^{fi}(M)}\mid N\text{ is semiprime }\}\cup\{M\}.\]
Then $SP(M)$ is a frame. Moreover, $SP(M)\cong \mathcal{O}(LgSpec(M))$ canonically as frames.
\end{prop}

\begin{proof}
By Proposition \ref{3}, $\Lambda^{fi}(M)_\mu=SP(M)$. Since $\mu$ is a multiplicative nucleus then $SP(M)$ is a frame by Corollary \ref{05a}.
\end{proof}

\begin{dfn}\label{j}
Let $L$ be a lattice. An element $1\neq{p}\in{L}$ is called $\wedge$\emph{-irreducible} if whenever $x\wedge{y}\leq{p}$ for any $x,y\in{L}$ then $x\leq{p}$ or $y\leq{p}$.
\end{dfn}

Given a frame $F,$ {\em its points} is the set 
\[pt(F)=\{p\in{F}\mid p\text{ is}\wedge\text{-irreducible}\}\]

\begin{prop}\label{77}
Let $M$ be an $R$-module. Then $pt(SP(M))=Spec(\Lambda^{fi}(M))$
\end{prop}

\begin{proof}
It is clear that $Spec(\Lambda^{fi}(M))\subseteq{pt(SP(M))}$. Now, let $P\in{pt(SP(M))}$ and $N,L\in{\Lambda^{fi}(M)}$ such that $N_ML\leq{P}$. By Proposition \ref{42},
\[\mu(N)\cap\mu(L)=\mu(N\cap{L})=\mu(N_ML)\leq\mu(P)=P.\]
Since $\mu(N),\mu(L)\in{SP(M)}$ then $N\leq\mu(N)\leq{P}$ or $L\leq\mu(L)\leq{P}$.
\end{proof}

\begin{rem}
Since $pt(SP(M))=Spec(\Lambda^{fi}(M)),$ there exists  a continuous function 
\[\eta\colon LgSpec(M)\rightarrow{Spec(\Lambda^{fi}(M))}\]

\noindent defined as $\eta(P)=\underset{SP}{\sum}\{N\in{SP(M)}\mid P\in\mathcal{V}(N)\}$, where $\underset{SP}{\sum}$ denotes the suprema in the frame $SP(M)$.
Also, we have a frame isomorphism between $SP(M)$ and $\mathcal{O}(Spec(\Lambda^{fi}(M)))$. See \cite{H5}
\end{rem}

\begin{dfn}\label{k}
Let $F$ be a frame. It is said that $F$ is \emph{spatial} if  it is isomorphic to $\mathcal{O}(X)$ for some topological space $X$.  If each  quotient frame of $F$ is spatial, it is said that $F$ is \emph{totally spatial}.
\end{dfn}

\begin{thm}\label{21}
Let $M$ be projective in $\sigma[M]$. Suppose for every fully invariant submodule $N\leq{M}$, the factor module $M/N$ has finite uniform dimension. Then, the frame $SP(M)$ is totally spatial.
\end{thm}

\begin{proof}
Let $N\in{SP(M)}$. By \cite[Lemma 9]{V} $M/N$ is projective in $\sigma[M/N]$. Since $N$ is semiprime in $M$ then $M/N$ is a semiprime module (Definition \ref{sp}) by \cite[Proposition 13]{S}. By hypothesis $M/N$ has finite uniform dimension, so by \cite[Proposition 1.29]{Gold} ${Spec(\Lambda^{fi}(M/N))}$ has finitely many minimal elements $(P_1/N),...,(P_n/N)$ such that $0=P_1/N\cap...\cap{P_n/N}$. 

Since $P_i/N\in{Spec(\Lambda^{fi}(M/N))}$ then $M/P_i\cong\frac{M/N}{P_i/N}$ is a prime module. Thus, by \cite[Proposition 18]{PI},  $P_i\in{Spec(\Lambda^{fi}(M))}$ for all $1\leq{i}\leq{n}$. Moreover, $N=P_1\cap...\cap{P_n}.$ Because of  this intersection is finite, we can assume that it is irredundant. Thus by \cite[Theorem 3.4]{Spa}, $SP(M)$ is totally spatial.
\end{proof}

For the definition of a \emph{weakly scattered space}, see \cite{H5} and  \cite[8.1]{Pi}.

\begin{cor}\label{22}
Let $M$ be projective in $\sigma[M]$. Suppose for every fully invariant submodule $N\leq{M}$, the factor module $M/N$ has finite uniform dimension. Then the topological space $LgSpec(M)$ is weakly scattered.
\end{cor}

\begin{proof}
By Proposition \ref{7} $SP(M)\cong\mathcal{O}(LgSpec(M))$, then $LgSpec(M)$ is weakly scattered by \cite[Corollary 3.7]{Spa}.
\end{proof}

For the definition of Krull dimension of a module $M$ see \cite{K}.

\begin{cor}
Let $M$ be projective in $\sigma[M]$. If $M$ has Krull dimension, then $LgSpec(M)$ is weakly scattered.
\end{cor}

\begin{proof}
If $M$ has Krull dimension then every factor module $M/N$ so does. Now, if $M/N$ has Krull dimension, by \cite[Proposition 2.9]{K} $M/N$ has finite uniform dimension. So by Corollary \ref{22} $LgSpec(M)$ is weakly scattered.
\end{proof}

\begin{rem}
Note that with the last corollary, for every commutative ring $R$ with Krull dimension, $Spec(R)$ is weakly scattered. In particular, for every noetherian commutative ring. Also  see \cite[Theorem 4.1]{Spa}.
\end{rem}

\begin{rem}\label{ob}
In \cite[Theorem 7.5]{H5}, for a $T_{0}$ space $S$, the frame of all nuclei over the topology of $S$ is boolean precisely when $S$ is scattered, so by \ref{22}, we observe that $N\mathcal{O}(LgSpec(M))\cong NSP(M)$ is boolean.
\end{rem}

If $R$ is a ring, the lowest radical of $R$ is defined as 
\[Nil_*(R)=\bigcap\{P\mid P\in Spec(Bil(R))  \}.\]
It is well known that $Nil_*(R)$ is contained in $Rad(R)$. 

\begin{dfn}\label{l}
Let $M\in{\RMod}$. The lowest radical of $M$ is the fully invariant submodule
\[Nil_*(M)=\bigcap\{Q\mid Q\in{Spec(\Lambda^{fi}(M))}\}.\]
\end{dfn}

\begin{prop}\label{8}
$Nil_*(M)\leq{Rad(M)}.$
\end{prop}

\begin{proof}
Since every maximal submodule is a large prime submodule by Remark \ref{4.23}, $Rad(M)$ is semiprime in $M$. So $Rad(M)$ is an intersection of elements of $Spec(\Lambda^{fi}(M))$ by \cite[Proposition 1.12]{Gold}. This implies that $Nil_*(M)\leq{Rad(M)}.$ 
\end{proof}

If 
\[Max(M)=\{\mathcal{M}<M\mid\mathcal{M}\text{ is a maximal submodule }\},\]
we have that $Max(M)\subseteq {LgSpec(M)}$. Notice that $Max(M)$ is a subspace of $LgSpec(M).$ 

Suppose that $M$ is coatomic. Then we have the adjunction 
\[\xymatrix{\Lambda^{fi}(M)\ar@/^/[r]^{\textit{m}} & \mathcal{O}(Max(M))\ar@/^/[l]^{\textit{m}_*}}\]
defined as 
\[\textit{m}(N)=\{\mathcal{M}\in{Max(M)}\mid N\nsubseteq\mathcal{M}\},\] and
\[\textit{m}_*(A)=\sum\{K\in{\Lambda^{fi}(M)}\mid\textit{m}(K)\subseteq{A}\}.\]
This adjunction can be factorized as
\[\xymatrix@=25mm{\Lambda^{fi}(M)\ \ \ar@/^/[r]^{\mathcal{U}}\ar@/^/[rd]^{\textit{m}} & \mathcal{O}(LgSpec(M))\ \ \ar@/^/[l]^{\mathcal{U}_*}\ar@/^/[d]^{I} \\ & \mathcal{O}(Max(M))\ar@/^/[ul]^{\textit{m}_*}\ar@/^/[u]^{I_*}} ,\]

\noindent where $I\colon\mathcal{O}(LgSpec(M))\rightarrow\mathcal{O}(Max(M))$ is given by $I(A)=A\cap{Max(M)}$.

Following the proof of Proposition \ref{41}, $(\textit{m}_*\circ\textit{m})(N)$ is the largest fully invariant submodule contained in 
\[\bigcap_{\mathcal{M}\in\mathcal{V}(N)\cap{Max(M)}}{\mathcal{M}}\]
Then 
\[(\textit{m}_*\circ\textit{m})(0)=\bigcap_{\mathcal{M}\in{Max(M)}}{\mathcal{M}}=Rad(M)\]

Note that the proof of Theorem \ref{3.15} can be applied to $\tau:=\textit{m}_*\circ\textit{m}$, so $\tau$ is a multiplicative nucleus. Hence $R(M):=\Lambda^{fi}(M)_\tau$ is a frame by Corollary \ref{05a}. Moreover, $R(M)$ is a subframe of $SP(M)$.

\begin{dfn}
Let $M$ be an $\R$-module. $M$ is {\em co-semisimple} if every simple module in $\sigma[M]$ is $M$-injective. 

If $M= {_{\R}\R}$, this is the definition of a {\em left $V$-ring. }
\end{dfn}

The following characterization of co-semisimple modules is given in \cite[23.1]{W}

\begin{prop}
For an $\R$-module $M$ the following statements are equivalent:
\begin{enumerate}
	\item $M$ is co-semisimple.
	\item Any proper submodule of $M$ is an intersection of maximal submodules.
\end{enumerate}
\end{prop} 
 
\begin{cor}
Let $M$ be a co-semisimple module. Then $\Lambda^{fi}(M)$ is a frame.
\end{cor}

\begin{proof}
Since $M$ is co-semisimple then $SP(M)=\Lambda^{fi}(M)$.
\end{proof}

\begin{cor}
Let $M$ be a duo co-semisimple module then $\Lambda(M)$ is a frame. 
\end{cor}

\begin{proof}
In this case $\Lambda(M)=\Lambda^{fi}(M)$.
\end{proof}

\section*{Acknowledgments}

The authors wish to thank Professor  Harold Simmons  for  sharing \cite{H} which inspired the topics in this paper. Finally, the authors are grateful for the referee's many valuable comments and helpful suggestions  which have improved this one.

\footnotesize
 \vskip2mm
\noindent	 \textsc{Mauricio Medina B\'arcenas$^{(\ast)}$,  Angel Zald\'ivar Corichi$^{\dag}$.}\\
Instituto de Matem\'aticas, Universidad Nacional Aut\'onoma de M\'exico, \'area de Investigaci\'n Cient\'ifica, Circuito Exterior, C.U., 04510, M\'exico, D.F., M\'exico.\vspace{5pt}

\noindent \textsc{Martha Lizbeth Shaid Sandoval Miranda$^{(\ddag)}$.}\\
Facultad de Ciencias, Universidad Nacional Aut\'onoma de M\'exico.\\
Circuito Exterior, Ciudad Universitaria, 04510, M\'exico, D.F., M\'exico.
\begin{itemize}
\item[ e-mails:]  $(\ast)$ mmedina@matem.unam.mx \;  {\it corresponding author}
 \item[\hspace{33pt}]  (\dag) zaldivar@matem.unam.mx 
\item[\hspace{33pt}] (\ddag) marlisha@ciencias.unam.mx
\end{itemize}
\end{document}